\documentclass[12pt]{article}%

\usepackage{amsfonts}%
\usepackage{amssymb}%
\usepackage{amsmath}%
\usepackage{amsthm}
\usepackage{array}
\usepackage{graphicx}
\usepackage[text={6.0in,8.5in}]{geometry}
\usepackage[colorlinks,citecolor=blue,urlcolor=blue]{hyperref}
\usepackage{thmtools}
\usepackage{tikz}
\usepackage{setspace}
\usepackage{subfig}
\usepackage{nicefrac}
\usepackage{natbib} 
\usepackage{multirow}

\DeclareMathOperator*{\argmax}{arg\,max}

\theoremstyle{definition}

\newcommand{\JEL}[1]{\par\noindent
{\small{\em \textsc{JEL Codes}\/}: #1}}

\newcommand{\Keywords}[1]{\par\noindent
{\small{\em \textsc{Keywords}\/}: #1}}

\newtheorem{theorem}{Theorem}

\newtheorem{corollary}{Corollary}

\newtheorem{definition}{Definition}

\newtheorem{lemma}{Lemma}

\begin{document}

\title{An Information-Theoretic Foundation\\for the Weighted Updating Model}

\author{Jesse Aaron Zinn\thanks{I appreciate Ted Bergstrom, Gary Charness, Itzhak Gilboa, Zack Grossman, Botond K\H{o}szegi, Jason Lepore, Dick Startz and seminar participants at Cal Poly San Luis Obispo, the 2014 Bay Area Behavioral and Experimental Economics Workshop, and the 2013 meetings of the Society for the Advancement of Behavioral Economics for valuable comments and suggestions.}}
\maketitle


\begin{abstract}
Weighted Updating generalizes Bayesian updating, allowing for biased beliefs by weighting the likelihood function and prior distribution with positive real exponents. I provide a rigorous foundation for the model by showing that transforming a distribution by exponential weighting (and normalizing) systematically affects the information entropy of the resulting distribution. For weights greater than one the resulting distribution has less information entropy than the original distribution, and vice versa. The entropy of a distribution measures how informative a decision maker is treating the underlying observation(s), so this result suggests a useful interpretation of the weights. For example, a weight greater than one on a likelihood function models an individual who is treating the associated observation(s) as being more informative than a perfect Bayesian would. 
\begin{singlespace}
\JEL{C02, D03, D83}
\end{singlespace}
\end{abstract}

\section{Introduction}

The weighted updating model generalizes Bayes' rule to allow for biased learning. Despite the fact that this model has seen some use in economics and other disciplines, there has not yet been a rigorous interpretation of the model or justification for using it. Those who use the model have, heretofore, justified its use by appealing to intuition. This paper eliminates this shortcoming with a result that provides a rationale for using the model as it has been used. Namely, I show that transforming a distribution by exponentially weighting it and normalizing systematically affects the information entropy of the resulting distribution relative to that of the original distribution.

Whether the information entropy of the resulting distribution is greater or less than the original distribution depends on the magnitude of the exponential weight. If the weight is greater than one then the information entropy of the resulting distribution is less than that of the original distribution, and vice versa. Each of the distributions constituting Bayes' rule represents how either empirical observations or prior information affects beliefs. As such, the information entropies of these distributions measure how informative the individual regards these pieces of information. Therefore, this result provides the interpretation that weighting is a parametric method with which to model the treatment of data as either more or less informative than with Bayesian updating. One can thus say that weighted updating embodies a theory of biased judgement, wherein these biases are a result of the treatment of data as containing inaccurate levels of information content.

Literature that uses the weighting updated model in one form or another includes \citet{grether1980bayes,grether1992testing}, who estimates the exponential weights on the likelihood function and the prior distribution to find empirical evidence for the representativeness heuristic. \cite{ibrahim2000power} introduced \textit{power priors}, which allows the researcher to consider data from previous studies by putting a weight in $(0,1)$ on the likelihood function for that data and putting a weight of $1$ for current data. \cite{van2009dynamic} define a \textit{weighted product updating rule} and go on to prove that Bayes' rule and the Jeffrey updating rule are both special cases. \cite{palfrey2012speculative} use weighted updating to model investor under- and overreaction to public information about financial assets in a model with speculative pricing. \cite{benjamin2013model} use the weighted updating model to study non-belief in the law of large numbers.

\section{The Weighted Updating Model}\label{sec:InterpretWeights}

A decision maker will consider an observation (or sequence of observations) $x$ as an outcome from a stochastic process with probability density function $f(x|\theta)$, where $\theta$ is an unknown parameter that the decision maker considers to be from parameter space $\Theta$. Bayesian beliefs regarding the value of $\theta$ after observing $x$ are completely described by the posterior distribution $\pi(\theta|x)$. If we let $(\Theta,\mathcal{A},m)$ be a measure space and denote the likelihood function with $f(x|\theta)$ and the prior distribution with $\pi(\theta)$, then Bayes' rule states that
\[
\pi(\theta|x) = \frac{f(x|\theta)\pi(\theta)}{\int_\Theta f(x|\theta) \pi(\theta) \, dm(\theta)}.
\]

Weighted updating augments Bayes' rule with real-valued parameters $\alpha$ and $\beta$ as exponents respectively on the likelihood function and prior probability distribution. Denote the posterior distribution under weighted updating after observing $x$  by $\tilde{\pi}(\theta|x)$. Then the weighted updating model is given by

\begin{equation}\label{eq:WeightedUpdating}
\tilde{\pi}(\theta|x) = \frac{f(x|\theta)^\beta\pi(\theta)^\alpha}{\int_\Theta f(x|\theta)^\beta\pi(\theta)^\alpha \, dm(\theta)}.
\end{equation}

Both Bayes' rule and the weighted updating model can be stated without mention of the marginal distribution, which is not a function of $\theta$ and serves only as a normalization, ensuring that the posterior distribution aggregates to one over its support.\footnote{Throughout the paper, I assume all functions are measurable and integrable so that integrals are finite and well-defined. This assumption includes the functions generated by exponential weighting. In many cases, this assumption is innocuous because weighting a distribution with an exponent and rescaling results in a distribution from the original family, so integrability follows. However, this does not always hold true. For example, the function $(1-p)x^{-p}$ represents a distribution over $x \geq 1$ if and only if $p > 1$. Taking such a distribution to a power $\alpha < 1/p$ and doing the usual normalization does not result in another distribution, as the integral over $[1,\infty)$ of the resulting function diverges.} Thus, the weighted updating model can be displayed as 
\begin{equation}\label{WeightedUpdating.2}
\tilde{\pi}(\theta|x) \propto f(x|\theta)^\beta \pi(\theta)^\alpha \tag{\ref{eq:WeightedUpdating}$'$}.
\end{equation}
Stating the model as in expression \eqref{WeightedUpdating.2} emphasizes how the nature of the posterior distribution depends solely on the interaction between the prior distribution and the likelihood distribution, and how the weights $\alpha$ and $\beta$ affect this interaction. 

\cite{zinn2015expanding} expands upon the basic model in expression \eqref{eq:WeightedUpdating} in two ways: The first involves individuals who may discriminate between observations. Given a sequence of observations $h_t = (x_1,\ldots,x_t)$, the version of \eqref{WeightedUpdating.2} that allows for such discrimination would be
\begin{equation*}\label{eq:Discrimination}
\tilde{\pi}(\theta|h_t) \propto \pi(\theta)^\alpha \prod_{j=1}^t f(x_j|h_{j-1}, \theta)^{\beta_j},
\end{equation*}
where $\alpha$ still weights the prior distribution, $\beta_j$ is the weight associated with the $j$th observation $x_j$, and $h_j = (x_1,\ldots,x_j)$, for each $j \in \{1,\ldots,t\}$.  The second expansion of the basic weighted updating model allows the weights to change over time, which is simply a matter of letting the weights be functions of time, as in
\[
\tilde{\pi}(\theta|h_t) \propto f(h_t|\theta)^{\beta(t)} \pi(\theta)^{\alpha(t)}.
\]
In addition to those biases that \eqref{eq:WeightedUpdating} is capable of modelling, these expansions allow weighted updating to model the availability heuristic, base-rate neglect, the law of small numbers, non-belief in the law of large numbers, order effects (e.g. recency and primacy), the representativeness heuristic, and self-attribution bias.

\section{Monotone Concentration and Dispersion}

In this section, I investigate how exponentially weighting a distribution affects the resulting distribution relative to the original. Note that I change the notation from what is used in the previous section in an effort to make it clear that the following results apply to the likelihood function(s) \textit{and} the prior distribution individually. Likewise, results involving the weight $\gamma$ can be interpreted as true all of the exponential weights mentioned in the previous section.

Let $(\Omega, \mathcal{S}, M$) be a measure space and consider the transformation
\begin{equation}\label{eq:transform}
g  \longmapsto  \frac{g(\omega)^\gamma}{\int_\Omega g(\omega)^\gamma \, dM(\omega)}.
\end{equation}
If $\gamma > 0$ then \eqref{eq:transform} is a strictly increasing transformation, which implies that the same is true for its inverse. As such, the value(s) that maximize (or minimize) $g$ and the distribution proportional to $g^\gamma$ are identical. In other words, transformation \eqref{eq:transform} is necessarily mode-preserving. Note, however, that such a transformation is not necessarily mean-preserving, as is likely to be the case when $g$ is an asymmetric distribution. 

Most relevant for this work is that the exponent $\gamma$ affects how concentrated or dispersed the resulting distribution is. The following definition precisely describes what I mean by ``concentrated" and "dispersed".\footnote{Note that ``monotone dispersion'' is distinct from ``monotone spread'', a related concept due to \cite{quiggin1988increasingrisk}. A monotone dispersion differs from a monotone spread in that the latter is necessarily mean-preserving.}

\begin{definition}[Monotone Dispersion, Monotone Concentration]\label{def:dispersion}
	For two non-uniform probability distributions $\Gamma$ and $g$ on the same support $\Omega$, $\Gamma$ is a \textit{monotone dispersion of} $g$ if for all pairs $(\omega_1, \omega_2) \in \Omega^2$ the following three conditions hold: 
	\begin{align}
	g(\omega_1) = g(\omega_2) \quad & \Leftrightarrow \quad \Gamma(\omega_1) = \Gamma(\omega_2), \label{eq:MonSp1}\\
	g(\omega_1) > g(\omega_2) \quad & \Leftrightarrow \quad \Gamma(\omega_1) > \Gamma(\omega_2), \textrm{ and} \label{eq:MonSp2}\\
	g(\omega_1) > g(\omega_2) \quad & \Rightarrow \quad \frac{g(\omega_1)}{g(\omega_2)} > \frac{\Gamma(\omega_1)}{\Gamma(\omega_2)} \label{eq:MonSp3}.
	\end{align}
	If $\Gamma$ is a monotone dispersion of $g$ then $g$ is a \textit{monotone concentration of} $\Gamma$.\footnote{Uniform distributions are excluded from Definition \ref{def:dispersion} because if either $g$ or $\Gamma$ were uniform then the other would necessarily be uniform by condition \eqref{eq:MonSp1}, so they would be the same distribution. If this is the case then conditions \eqref{eq:MonSp2} and \eqref{eq:MonSp3} are only \textit{vacuously} true, which is not useful for our purposes because condition \eqref{eq:MonSp3} provides an asymmetry that allows one to compare different distributions. Excluding uniform distributions ensures that there is no case in which the relations ``is a monotone dispersion of'' and ``is a monotone concentration of'' are symmetric.}
\end{definition}

See Figure \ref{fig:MonDispersionConc} for an example of two distributions that are related to each other through monotone dispersion and concentration. Note that conditions \eqref{eq:MonSp1} and \eqref{eq:MonSp2} require that the transformations $g \mapsto \Gamma$ and $\Gamma \mapsto g$ are strictly increasing functions. Such monotonicity ensures that the ordinal properties are identical within pairs of distributions that are dispersions and concentrations of each other. In other words, two agents with beliefs that are related by monotone dispersion or concentration will agree on a rank ordering of events according to their likelihoods as given by their respective beliefs. Expression \eqref{eq:MonSp3} describes how the cardinal properties of a monotone dispersion or concentration differ from the original function, with a monotone dispersion being closer to a uniform distribution. Thus, a monotone concentration, with ``higher highs'' and ``lower lows'', is an exaggeration of any of its monotone dispersions. The following theorem states these notions rigorously.\footnote{Proofs for all results are in the appendix.}

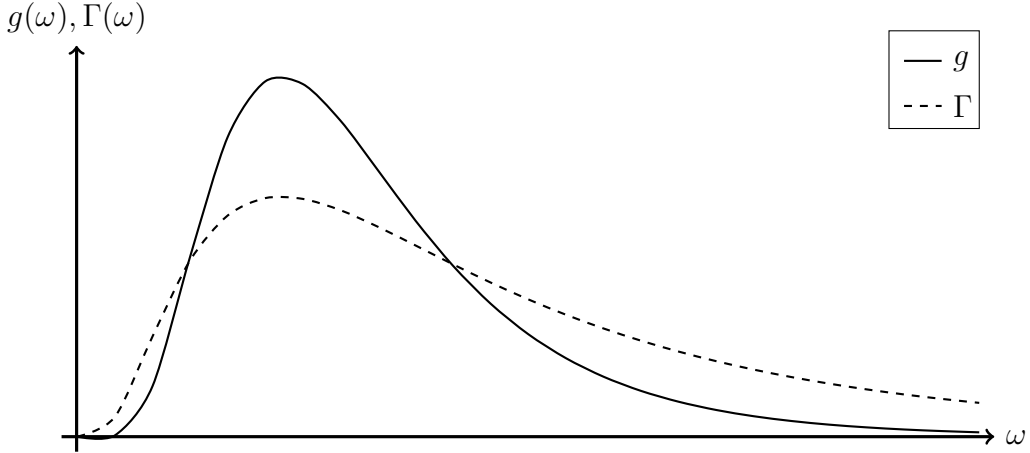
\begin{figure}
	\caption[Monotonic Dispersion and Concentration]{$\Gamma$ is a monotone dispersion of $g$, $g$ is a monotone concentration of $\Gamma$.}
	\label{fig:MonDispersionConc}
	\begin{center}
		\begin{tikzpicture}[smooth, scale = 1, domain=0.01:12]
		\draw                  (10.8,4.1) rectangle (12,5.4);
		\draw [thick]          (11,5) -- (11.5,5) node[right] {$g$};
		\draw [dashed, thick]   (11,4.4) -- (11.5,4.4) node[right]  {$\Gamma$};
		
		\draw[very thick, ->]  (-0.2,0)  --  (12.2,0)  node[right]  {$\omega$};
		\draw[very thick, ->]  (0,-0.2)  --  (0,5.2)  node[above]  {$g(\omega), \Gamma(\omega)$};
		
		\draw [thick]       plot  (\x,{6/(0.5*sqrt(2*pi)) * exp(-(ln(\x) - 1)^2/0.5)});
		\draw [dashed, thick] plot  (\x,{6/(0.75*sqrt(2*pi)) * exp(-(ln(\x) - 1)^2/1.125)});
		
		\end{tikzpicture}
	\end{center}
	
\end{figure}

\begin{theorem}\label{thm:HigherHighs}
	Let $\Gamma$ be a monotone dispersion of $g$. For any $\omega_1, \omega_2 \in \Omega$, 
	\[
	g(\omega_1) >g(\omega_2) \geq \Gamma(\omega_2) \quad \Rightarrow \quad g(\omega_1) > \Gamma(\omega_1).
	\]
	Also, 
	\[
	g(\omega_1) < g(\omega_2) \leq \Gamma(\omega_2) \quad \Rightarrow \quad g(\omega_1) < \Gamma(\omega_1).
	\]
\end{theorem}

The next theorem guarantees that if a distribution has a maximum then that maximum is no less than that of any associated monotone dispersion.
\begin{theorem}\label{cor:ConcMaxLarger}
	Let $\Gamma$ be a monotone dispersion of $g$ and let $\omega^*$ be a maximizer of $g$. Then $g(\omega^*) \geq \Gamma(\omega^*)$. If it is also the case that 
	\[
	M(\{\omega: g(\omega) < g(\omega*)\}) > 0,
	\]
	then $g(\omega^*) > \Gamma(\omega^*)$.
\end{theorem}

Interest in monotone dispersions and concentrations is due to the fact that when a distribution is weighted with a positive power and normalized, the resulting distribution is either a monotone dispersion or concentration of the original distribution depending on whether the weight is less than or greater than one, as stated in Theorem \ref{prop:PowerDispersion} below.

\begin{theorem}\label{prop:PowerDispersion}
	Let $g : \Omega \rightarrow \mathbb{R}_{++}$ be any non-uniform probability distribution. If $\gamma \in (0,1)$ then  
	\begin{equation}\label{eq:PowerDispersion}
	\frac{g(\omega)^\gamma}{\int_\Omega g(\omega)^\gamma \, dM(\omega)},
	\end{equation}
	is a monotone dispersion of $g$. If it is the case that $\gamma > 1$ then \eqref{eq:PowerDispersion} is a monotone concentration of $g$.
\end{theorem}

As long as both distributions are well-defined, Theorem \ref{prop:PowerDispersion} guarantees that a positively weighted distribution results in either a monotone dispersion or concentration of the original.

\section{How Should Dispersion be Measured?}

In this section, I study how the notions of dispersion and concentration from Definition \ref{def:dispersion} can be measured. I consider variance and show why it is not, in general, up to the task. In contrast, a distribution's information entropy is. I provide this result in Theorem \ref{thm:Entropy}, below, before discussing its implications for the weighted updating model.

\subsection{Variance?}

As variance is a widely used measure of how disperse a distribution is, one may suspect that a monotone dispersion would always have more variance than an associated monotone concentration. For many families of distributions this is indeed the case. Consider the normal distribution with mean $\mu$ and variance $\sigma^2$. It is straightforward to find that taking this distribution to the power $\gamma > 0$ results in a function that is proportional to the normal distribution with mean $\mu$ and variance $\sigma^2/\gamma$. After doing this manipulation, one can utilize Theorem \ref{prop:PowerDispersion} to note that $\gamma < 1$ leads to a monotone dispersion of the original distribution with greater variance and that a monotone concentration with less variance is the outcome if $\gamma > 1$.

Despite being true for the normal distribution (and many others), it is not the case for all distributions that a monotone dispersion implies greater variance and that a monotone concentration has less variance. Consider the beta distribution $B(a, b)$, which is proportional to $y^{a - 1}(1-y)^{b - 1}$ for $y \in [0,1]$ and parameters $a, b > 0$. Cases in which $a, b \in (0,1)$ result in a \textit{u}-shaped, strictly convex distribution with peaks at the extremes of the support, where $y = 0, 1$. Applying a monotone dispersion results in a more uniform distribution with less variance and applying a monotone concentration shifts mass toward the end-points of $[0,1]$, resulting in greater variance. For a more concrete example suppose $a = b = \nicefrac{3}{4}$. In this case, $B(\nicefrac{3}{4},\nicefrac{3}{4})$ has a variance of $\nicefrac{1}{10}$. Raising this distribution to the power $\gamma = 2$ and normalizing yields $B(\nicefrac{1}{2},\nicefrac{1}{2})$, the variance of which is $\nicefrac{1}{8}$. By Theorem \ref{prop:PowerDispersion}, $B(\nicefrac{1}{2},\nicefrac{1}{2})$ is a monotone concentration of $B(\nicefrac{3}{4},\nicefrac{3}{4})$, yet the former has a greater variance. This provide a counter-example against the general statement that a monotone concentration has less variance.

Variance does not have a consistent relationship with monotone dispersions and concentrations because it is a measure of dispersion \textit{from the mean of the distribution}. For a consistent relation with monotone dispersion and concentration it is necessary to have a measure of dispersion that is independent of reference points. Information entropy is a measure of dispersion that does not depend on reference points, so we consider it in the next subsection.

\subsection{Information Entropy?}

As will be shown subsequently in Theorem \ref{thm:Entropy}, a distribution's information entropy is a measure of dispersion or uncertainty that is invariably greater for monotone dispersions and less for monotone concentrations. Before getting to that result, I will introduce and briefly discuss the concept of information entropy.

\begin{definition}[Information Entropy, \citep{shannon1948mathematical}]\label{def:entropy}
	For any distribution $g: \Omega \rightarrow \mathbb{R}_{++}$, the \textit{information entropy of} $g$ is given by
	\[
	H(g) \equiv -\int_\Omega g(\omega) \log g(\omega) \, dM(\omega).
	\]
\end{definition}

The logarithm ensures that information entropy is additive in the densities of independent random variables. This is because for any two independent random variables $Y$ and $Z$ respectively distributed $g_Y$ and $g_Z$ and particular pair of events $(y, z)$,
\[
-\log g_Y(y)g_Z(z) = -\log g_Y(y) - \log g_Z(z).
\]

\cite{tribus1961thermostatics} dubbed  $-\log g(\omega)$ the \textit{surprisal} of $\omega$, for any distribution $g$ and particular $\omega \in \Omega$. Because $-\log g(\omega)$ is decreasing in $g(\omega)$, surprisal is greater for $\omega$ which (according to $g$) are less likely and, therefore, more \textit{surprising} outcomes. The information entropy of a distribution is equivalent to the \textit{expected surprisal}, as information entropy is the expected value of the surprisal of a distribution. If outcomes from one distribution are, on average, more surprising than outcomes from another distribution, then the first distribution can be thought of as containing less information than the second. Thus, distributions with greater information entropy will, on average, generate observations that have less information content, and vice versa.\footnote{The interpretation of information entropy as a measure of the uninformativeness of a distribution is consistent with the idea that physical entropy, which is proportional to information entropy by Boltzmann's constant, is a measure of one's \textit{ignorance} of a system. See, for example, the discussion in \citet[\S 5.3]{sethna2006entropy} for this interpretation of physical entropy along with a discussion of its relationship with information entropy.}

The following theorem verifies the claim that the information entropy of a monotone dispersion is greater than the information entropy of a monotone concentration.

\begin{theorem}\label{thm:Entropy}
	If $\Gamma$ is a monotone dispersion of $g$ then $H(\Gamma) \geq H(g)$. If, in addition, 
	\[
	M(\{\omega: \Gamma(\omega) \neq r\}) > 0
	\]\
	for some $r \in [b,B]$, where
	\[
	b \equiv \sup \Gamma(\{\omega : g(\omega) \leq \Gamma(\omega)\}) \in \mathbb{R}
	\]
	and
	\[
	B \equiv \inf \Gamma(\{\omega : g(\omega) \geq \Gamma(\omega)\}) \in \mathbb{R},
	\]
	then $H(\Gamma) > H(g)$.
\end{theorem}

\begin{corollary}\label{corollary:MainResult}
	If $g$ and $\frac{g(\omega)^\gamma}{\int_\Omega g(\omega)^\gamma \, dM(\omega)}$ are distributions for some $\gamma > 0$ then
	\begin{align*}
	H(g) & > H\left(\frac{g(\omega)^\gamma}{\int_\Omega g(\omega)^\gamma \, dM(\omega)}\right) \quad \Leftrightarrow \quad \gamma >1,\textrm{ and}\\ 
	H(g) & < H\left(\frac{g(\omega)^\gamma}{\int_\Omega g(\omega)^\gamma \, dM(\omega)}\right) \quad \Leftrightarrow \quad \gamma <1.
	\end{align*}
\end{corollary}

The preceding corollary applies Theorem \ref{thm:Entropy} under the special case where the monotone dispersion or concentration is a result of exponential weighting. It is essentially a syllogism with Theorems \ref{prop:PowerDispersion} and \ref{thm:Entropy} as premises.

Corollary \ref{corollary:MainResult} provides a rationale and rigorous interpretation of the weighted updating model. Consider expression \eqref{eq:WeightedUpdating}. Corollary \ref{corollary:MainResult} implies that if $\alpha > 1$ then the prior information is being overemphasized and treated with more information content than it should, and the opposite holds if $\alpha <1$. The same interpretation holds for $\beta$ with respect to the how informative $x$ is treated. So if, say, $\alpha > 1$ and $\beta < 1$ then the beliefs represented by that model are made in a manner wherein the prior information is treated with more information content than it should be and $x$ is treated as though it has too little information content.

\section{Concluding Remarks}\label{sec:Conclusion}

In this paper, I provide an interpretation of weighted updating as a method of modelling individuals who treat information as either more or less informative than under Bayes' rule. In particular, I show that weighting the functions primitive to Bayes' rule transforms the functions by monotone dispersion or monotone concentration, and that these transformations systematically affect the information entropies of the resulting likelihood function(s) and prior probability distribution.

This interpretation of weighting a distribution suggests that, on its own, weighted updating may be appropriate to model only those biases in which individuals correctly interpret information, but for some reason do not use the information in a rational way. Thus, for example, weighted updating may be utilized to model biases based on self-deception\footnote{Self-deception typically involves individuals who downplay or overemphasize the importance of certain pieces of evidence in a systematic way \citep{hirshleifer2001investor}.} or the cognitive limitations of utilizing correctly interpreted data. However, it may not be appropriate for modelling the type of confirmation bias studied by \cite{rabin1999first}, which involves decision makers who misinterpret information. Still, there is no reason why there should be only one type of bias affecting belief formation; one could, for example, model individuals who misinterpret evidence using the framework of \cite{rabin1999first} and then process the misinterpreted information irrationally using weighted updating.

\appendix

\section*{Appendix: Proofs}\label{[Appendix]}

\begin{proof}[Proof of Theorem \ref{thm:HigherHighs}]
	Let $g(\omega_1) >g(\omega_2) \geq \Gamma(\omega_2)$. As $\Gamma$ is a monotone dispersion of $g$, $g(\omega_1) >g(\omega_2)$ implies
	\[
	\frac{g(\omega_1)}{g(\omega_2)} > \frac{\Gamma(\omega_1)}{\Gamma(\omega_2)},
	\]
	which can be rearranged to obtain
	\[
	\frac{\Gamma(\omega_2)}{g(\omega_2)} > \frac{\Gamma(\omega_1)}{g(\omega_1)}.
	\]
	Now utilize $g(\omega_2) \geq \Gamma(\omega_2)$ to augment the above inequality to obtain
	\[
	1 \geq \frac{\Gamma(\omega_2)}{g(\omega_2)} > \frac{\Gamma(\omega_1)}{g(\omega_1)}.
	\]
	And so, $g(\omega_1) > \Gamma(\omega_1)$. The other case implying the opposite conclusion is symmetric.
\end{proof}

\begin{proof}[Proof of Theorem \ref{cor:ConcMaxLarger}]
	Let $\omega^* \in \argmax_{\omega \in \Omega} g(\omega)$. The hypothesis that $\Gamma$ is a monotone dispersion of $g$ implies that both $\Gamma$ and $g$ are non-uniform, so 
	\[
	\{\omega: g(\omega) < g(\omega^*)\} \neq \emptyset.
	\]
	For any $\omega_0 \in \{\omega: g(\omega) < g(\omega*)\}$, condition \eqref{eq:MonSp3} from Definition \ref{def:dispersion} indicates that 
	\begin{equation}\label{eq:MaxGT}
	\frac{g(\omega_0)}{g(\omega^*)} < \frac{\Gamma(\omega_0)}{\Gamma(\omega^*)}.
	\end{equation}
	For any $\omega^{**} \in \argmax_{\omega \in \Omega} g(\omega)$ it is the case that $g(\omega^{**}) = g(\omega^*)$. Therefore, condition \eqref{eq:MonSp1} tells us
	\begin{equation}\label{eq:MaxET}
	\frac{g(\omega^{**})}{g(\omega^*)} = \frac{\Gamma(\omega^{**})}{\Gamma(\omega^*)}.
	\end{equation}
	Since 
	\[
	\{\omega: g(\omega) < g(\omega^*)\} \, \cup \, \argmax_{\omega \in \Omega} g(\omega) = \Omega, 
	\]
	expressions \eqref{eq:MaxGT} and \eqref{eq:MaxET} imply
	\[
	\frac{g(\omega)}{g(\omega^*)} \leq \frac{\Gamma(\omega)}{\Gamma(\omega^*)} \quad \textrm{ for all } \omega \in \Omega.
	\]
	Integrating over $\Omega$ yields
	\begin{equation}\label{eq:Integral}
	\frac{\int_\Omega g(\omega) \, dM(\omega)}{g(\omega^*)} \leq \frac{\int_\Omega \Gamma(\omega) \, dM(\omega)}{\Gamma(\omega^*)}.
	\end{equation}
	As both $g$ and $\Gamma$ are probability distributions, the value of the integrals in the numerators of expression \eqref{eq:Integral} is one. Thus,
	\begin{equation*}\label{eq:AfterIntegral}
	\frac{1}{g(\omega^*)} \leq \frac{1}{\Gamma(\omega^*)},
	\end{equation*}
	which is true only if $g(\omega^*) \geq \Gamma(\omega^*)$. If it is also the case that \[
	M(\{\omega: g(\omega) < g(\omega^*)\}) > 0,
	\] 
	then the inequality in expression \eqref{eq:Integral} is strict, from which $g(\omega*) > \Gamma(\omega*)$ follows.
\end{proof} 

\begin{proof}[Proof of Theorem \ref{prop:PowerDispersion}]
	Let $\gamma \in (0,1)$. Conditions \eqref{eq:MonSp1} and \eqref{eq:MonSp2} are satisfied immediately, as a consequence of the fact that raising to a positive exponent is a monotonically increasing transformation. As $g$ is non-uniform there exists a pair $(\omega_1, \omega_2) \in \Omega^2$ for which $g(\omega_1) > g(\omega_2) > 0$. For any such pair, multiplying each term of $0 < \gamma < 1$ by $\log (g(\omega_1)/g(\omega_2))$ yields
	\[
	0 < \gamma\log \frac{g(\omega_1)}{g(\omega_2)} < \log \frac{g(\omega_1)}{g(\omega_2)},
	\]
	which implies that
	\[
	1 < \frac{g(\omega_1)^\gamma}{g(\omega_2)^\gamma} < \frac{g(\omega_1)}{g(\omega_2)}.
	\]
	Dividing both the numerator and denominator of the center term by the normalizing factor $\int_\Omega g(\omega)^\gamma \, dM(\omega) > 0$ yields,
	\[
	1 < \frac{g(\omega_1)^\gamma / \int_\Omega g(\omega)^\gamma \, dM(\omega)}{g(\omega_2)^\gamma / \int_\Omega g(\omega)^\gamma \, dM(\omega)} < \frac{g(\omega_1)}{g(\omega_2)},
	\]
	showing that condition \eqref{eq:MonSp3} is satisfied. The case for $\gamma > 1$ yielding a monotone concentration is proved analogously.
\end{proof}

The proof of Theorem \ref{thm:Entropy} utilizes the following lemmas.

\begin{lemma}\label{lem:NonEmpty}
	If $g$ and $\Gamma$ are distributions with the same support $\Omega$ then both of the sets
	\[
	\{\omega:g(\omega) \leq \Gamma(\omega)\} \qquad \textrm{and} \qquad \{\omega:g(\omega) \geq \Gamma(\omega)\}
	\]
	are nonempty.\footnote{I doubt that this result is original to this work. I include it, however, for sake of completeness.}
\end{lemma}

\begin{proof}
	Suppose $\{\omega:g(\omega) \leq \Gamma(\omega)\} = \emptyset$. Then $g(\omega) > \Gamma(\omega)$ for all $\omega \in \Omega$. This implies
	\[
	\int_\Omega g(\omega) \, dM(\omega) > \int_\Omega \Gamma(\omega) \, dM(\omega),
	\] 
	which, in turn, implies that $1>1$. This contradiction means 
	\[
	\{\omega:g(\omega) \leq \Gamma(\omega)\} \neq \emptyset.
	\]
	Showing that $\{\omega:g(\omega) \geq \Gamma(\omega)\} \neq \emptyset$ follows by simply transposing $g$ and $\Gamma$.
\end{proof}

Recall from Theorem \ref{thm:Entropy} the following definitions which are used in the next several results:
\[
b \equiv \sup \Gamma(\{\omega : g(\omega) \leq \Gamma(\omega)\}) \in \mathbb{R}.
\]
\[
B \equiv \inf \Gamma(\{\omega : g(\omega) \geq \Gamma(\omega)\}) \in \mathbb{R}.
\]

\begin{lemma}\label{lem:InR}
	If $g$ is a monotone concentration of $\Gamma$ then $b, B \in \mathbb{R}$.
\end{lemma}

\begin{proof}
	Due to Lemma \ref{lem:NonEmpty}, there exist $\omega_1 \in \{\omega : g(\omega) \leq \Gamma(\omega)\}$ and $\omega_2 \in \{\omega : g(\omega) \geq \Gamma(\omega)\}$. Suppose $\Gamma(\{\omega : g(\omega) \leq \Gamma(\omega)\})$ is not bounded above. Then, we can choose $\omega_1$ such that $\Gamma(\omega_1) > g(\omega_2)$. To summarize, we have
	\[
	\Gamma(\omega_1) > g(\omega_2) \geq \Gamma(\omega_2) > 0
	\]
	and
	\[
	\Gamma(\omega_1) \geq g(\omega_1).
	\]
	As $g$ is a monotone concentration of $\Gamma$, condition \eqref{eq:MonSp2} stipulates that \[
	\Gamma(\omega_1) > \Gamma(\omega_2) \quad \Leftrightarrow \quad g(\omega_1) > g(\omega_2).
	\]
	Therefore, 
	\[
	\Gamma(\omega_1) \geq g(\omega_1) > g(\omega_2) \geq \Gamma(\omega_2) > 0.
	\]
	This implies
	\[
	\frac{\Gamma(\omega_1)}{\Gamma(\omega_2)} > \frac{g(\omega_1)}{g(\omega_2)} \qquad \textrm{and} \qquad g(\omega_1) > g(\omega_2),
	\]
	contradicting condition \eqref{eq:MonSp3}, which holds since $g$ is a monotone concentration of $\Gamma$. Therefore, $\Gamma(\omega_1) \leq g(\omega_2)$ and, since $\omega_1$ was chosen arbitrarily, this holds for all $\omega_1 \in \{\omega : g(\omega) \leq \Gamma(\omega)\}$. Hence, $g(\omega_2)$ is an upper bound of $\Gamma(\{\omega : g(\omega) \leq \Gamma(\omega)\} \subset \mathbb{R}$. The completeness of $\mathbb{R}$, thereby, implies that $b \in \mathbb{R}$. That $B \in \mathbb{R}$ can be shown similarly, though it can be a bit simpler because $\Gamma(\{\omega : g(\omega) \geq \Gamma(\omega)\}) > 0$.
\end{proof}

\begin{lemma}\label{lem:bB}
	If $g$ is a monotone concentration of $\Gamma$ then $b \leq B$.
\end{lemma}

\begin{proof}
	Suppose that $B < b$. Then, use the definition of the supremum of a set to establish that there exists $\omega_1 \in \Omega$ for which 
	\begin{equation}\label{eq:bB1}
	g(\omega_1) < \Gamma(\omega_1) \qquad \textrm{and} \qquad \frac{b+B}{2} < \Gamma(\omega_1) \leq b.
	\end{equation}
	Similarly, the definition of the infimum of a set guarantees that there exists $\omega_2 \in \Omega$ such that
	\begin{equation}\label{eq:bB2}
	\Gamma(\omega_2) < g(\omega_2)   \qquad \textrm{and} \qquad  B \leq \Gamma(\omega_2) < \frac{b+B}{2}.
	\end{equation}
	Expressions \eqref{eq:bB1} and \eqref{eq:bB2} imply that $\Gamma(\omega_2) < \Gamma(\omega_1)$, which is true if and only if $g(\omega_2) < g(\omega_1)$ since $g$ be a monotone concentration of $\Gamma$. Therefore, 
	\[
	\Gamma(\omega_1) \geq g(\omega_1) > g(\omega_2) \geq \Gamma(\omega_2) > 0,
	\]
	which, as shown in the proof to Lemma \ref{lem:InR}, is inconsistent with condition \eqref{eq:MonSp3}. Therefore $b \leq B$.
\end{proof}

\begin{lemma}\label{lem:r}
	If $g$ is a monotone concentration of $\Gamma$ then there exists $r \in \mathbb{R}$ such that 
	\begin{equation}\label{eq:Lemma4.1}
	\Gamma(\omega) > r \quad \Rightarrow \quad g(\omega) > \Gamma(\omega)
	\end{equation}
	and
	\begin{equation}\label{eq:Lemma4.2}
	\Gamma(\omega) < r \quad \Rightarrow \quad g(\omega) < \Gamma(\omega).
	\end{equation}
\end{lemma}

\begin{proof}
	I prove \eqref{eq:Lemma4.1}, as the proof of \eqref{eq:Lemma4.2} is essentially identical. Lemma \ref{lem:bB} guarantees that $[b,B] \neq \emptyset$, so we may consider any $r \in [b,B]$. Choose $\omega_0 \in \Omega$ such that $\Gamma(\omega_0) > r$. Then 
	\[
	b \leq r < \Gamma(\omega_0).
	\]
	$\Gamma(\omega_0) > b$ implies that $\omega_0 \notin \{\omega:g(\omega)\leq\Gamma(\omega)\}$, so $g(\omega_0) > \Gamma(\omega_0)$.
\end{proof} 

\begin{proof}[Proof of Theorem \ref{thm:Entropy}]
	Since the $\log$ function is concave, Jensen's inequality implies
	\begin{equation}\label{eq:Kullback–Leibler}
	\int_\Omega g(\omega) \log \frac{\Gamma(\omega)}{g(\omega)} \, dM(\omega) \leq \log\int_\Omega g(\omega)  \frac{\Gamma(\omega)}{g(\omega)} \, dM(\omega). 
	\end{equation}
	The right-hand side of expression \eqref{eq:Kullback–Leibler} equals zero since, using the fact that $\Gamma$ is a distribution and integrates to unity, $\log\int_\Omega \Gamma(\omega) \, dM(\omega) = \log 1$. Therefore, expression \eqref{eq:Kullback–Leibler} implies Gibbs' inequality:\footnote{Gibbs' inequality is a well known fact from statistical physics. Previous versions of this proof initiated with Gibbs' inequality rather than Jensen's inequality. I decided to change this because the education of economists typically involves exposure to Jensen's inequality and not Gibbs' inequality.}
	\begin{equation*}\label{eq:Gibbs1}
	\int_\Omega g(\omega) \log \Gamma(\omega) \, dM(\omega) \leq \int_\Omega g(\omega) \log g(\omega) \, dM(\omega),
	\end{equation*}
	which can also be expressed
	\begin{equation}\label{eq:Gibbs2}
	\int_\Omega g(\omega) \log \Gamma(\omega) \, dM(\omega) \leq -H(g).
	\end{equation}
	Add $H(\Gamma)$ to both sides of expression \eqref{eq:Gibbs2} to find that
	\begin{equation*}\label{eq:Gibbs3}
	H(\Gamma) + \int_\Omega g(\omega) \log \Gamma(\omega) \, dM(\omega) \leq H(\Gamma)-H(g),
	\end{equation*}
	implying that 
	\begin{equation}\label{eq:Gibbs4}
	\int_\Omega [g(\omega)-\Gamma(\omega)] \log \Gamma(\omega) \, dM(\omega) \leq H(\Gamma)-H(g),
	\end{equation}
	Lemma \ref{lem:bB} asserts that $[b,B]$ is non-empty, so we can choose any $r \in [b, B]$ and write:
	\begin{equation}\label{eq:zeros}
	-\log r \int_\Omega g(\omega) - \Gamma(\omega) \, dM(\omega) = 0,
	\end{equation}
	which is true because  $g$ and $\Gamma$ are both distributions. Adding expressions \eqref{eq:Gibbs4} and \eqref{eq:zeros} and some rearranging yields
	\begin{equation}\label{eq:Finale}
	\int_\Omega [g(\omega)-\Gamma(\omega)] (\log \Gamma(\omega) -\log r) \, dM(\omega) \leq H(\Gamma)-H(g).
	\end{equation}
	By Lemma \ref{lem:r}, $r \in [b, B]$ implies that $\log \Gamma(\omega) -\log r$ and $g(\omega)-\Gamma(\omega)$ have the same sign, so the left-hand side of expression \eqref{eq:Finale} is non-negative. Therefore, $H(g) \leq H(\Gamma)$. If, additionally, $M(\{\omega \in \Omega: \Gamma(\omega) \neq r\}) > 0$ then the integral constituting the left-hand side of expression \eqref{eq:Finale} can be decomposed into
	\[
	\int\displaylimits_{\{\omega: \Gamma(\omega) \neq r\}} [g(\omega)-\Gamma(\omega)] (\log \Gamma(\omega) -\log r) \, dM(\omega)
	\]
	plus
	\[
	\int\displaylimits_{\{\omega: \Gamma(\omega) = r\}} [g(\omega)-\Gamma(\omega)] (\log \Gamma(\omega) -\log r) \, dM(\omega).
	\]
	The former, by Lemma \ref{lem:r}, is strictly positive and the latter equals zero, so $H(g) < H(\Gamma)$. 
\end{proof}

\bibliography{arXiv_Information.Theoretic.Foundation}
\end{document}